\documentclass{article}

\usepackage{amssymb,amsmath,color}
\usepackage{amsthm}
\usepackage{url}

\definecolor{red}{rgb}{1,0,0}
\definecolor{blue}{rgb}{.2,.2,.8}

\newtheorem{theorem}{Theorem}
\newtheorem{corollary}{Corollary}
\newtheorem{proposition}{Proposition}

\newtheorem{lemma}{Lemma}

\theoremstyle{remark}
\newtheorem{example}{Example}

\newcommand*\samethanks[1][\value{footnote}]{\footnotemark[#1]}

\begin{document}

\title{Inequalities involving the generating function for the number of 
partitions into odd parts}
\author{Cristina Ballantine\thanks{This work was partially supported by a grant from the Simons Foundation (\#245997 to Cristina Ballantine).}\\
\footnotesize Department of Mathematics and Computer Science\\
\footnotesize College of The Holy Cross\\
\footnotesize Worcester, MA 01610, USA \\
\footnotesize cballant@holycross.edu
\and Mircea Merca\samethanks
 \\ 
\footnotesize Department of Mathematics\\
\footnotesize University of Craiova\\
\footnotesize Craiova, DJ 200585, Romania\\
\footnotesize mircea.merca@profinfo.edu.ro
}
\date{}

\maketitle
\begin{abstract}
Fibonacci numbers can be expressed in terms of multinomial coefficients as sums over integer partitions into odd parts. We use this fact to introduce a  family of double inequalities involving the generating function for the number of partitions into odd parts and the generating function for the number of odd divisors.
\\
\\
{\bf Keywords:} integer partitions, Fibonacci numbers, multinomial coefficients
\\
\\
{\bf MSC 2010:} 05A20, 05A19, 05A17, 11B39
\end{abstract}

\section{Introduction}

The classical multinomial expansion is given by
$$(x_1+x_2+\cdots+x_n)^k=\sum\binom{k}{t_1,t_2,\ldots,t_n}x_1^{t_1}x_2^{t_2}\cdots x_n^{t_n},$$
where the sum is over all nonnegative integer solutions to the equation
$$t_1+t_2+\cdots+t_n=k$$
and
$$\binom{k}{t_1,t_2,\ldots,t_n}=\frac{k!}{t_1!t_2!\cdots t_n!}.$$
is  a multinomial coefficient. According to Fine \cite{Fine59}, we have the following connection between the binomial coefficients and multinomial coefficients:
\begin{equation} \label{eq1}
\binom{n-1}{k-1,n-k}=\sum_{\substack{t_1+2t_2+\cdots+nt_n=n\\t_1+t_2+\cdots+t_n=k}}\binom{k}{t_1,t_2,\ldots,t_n}, 
\end{equation}
with $n,k>0$. Recently, using this Fine's formula, Merca \cite{Merca14b} obtained new upper bounds for the number of partitions of $n$ into $k$ parts.

Recall that a composition of a positive integer $n$ is a way of writing $n$ as a sum of positive integers, i.e.,
\begin{equation*}\label{eq:1.1}
n=\lambda_1+\lambda_2+\cdots+\lambda_k.
\end{equation*}
When the order of integers $\lambda_i$ does not matter, this representation is known as an integer partition \cite{Andrews76} and can be rewritten as
$$n=t_1+2t_2+\cdots+nt_n,$$
where each positive integer $i$ appears $t_i$ times. The number of parts of this partition is given by
$$t_1+t_2+\cdots+t_n=k.$$
Throughout the paper, we denote by $q(n)$ the number of integer partitions of $n$ with odd parts. It is well-known that this equals the 
 number of integer partitions of $n$ into distinct parts.  We define $q(0) = 1$ and $q(n) = 0$ for any negative integer $n$.

\begin{example}
The integer $6$ has four partitions into odd parts:
$$1+1+1+1+1+1\ =\ 1+1+1+3\ =\ 1+5\ =\ 3+3$$
and four partitions into distinct parts:
$$1+2+3\ =\ 1+5\ =\ 2+4\ =\ 6.$$
Thus, $q(6)=4$.
\end{example}

The sequence $\{F_n\}_{n\geqslant 0}$ of Fibonacci numbers is defined by the recurrence relation
\begin{equation}\label{eq2}
F_n - F_{n-1} - F_{n-2} = 0,
\end{equation}
with seed values
$$F_0 = 0 \qquad\text{and}\qquad F_1 = 1\ .$$
Recently, Merca \cite[Corollary 9]{Merca14a} proved that the Fibonacci numbers can be expressed in terms of multinomial coefficients as sums over integer partitions into odd parts, i.e.,
\begin{equation}\label{eq3}
F_n=\sum_{t_1a_1+t_2a_2+\cdots+t_{\left\lceil n/2 \right\rceil }a_{\left\lceil n/2 \right\rceil }=n}\binom{t_1+t_2+\cdots+t_{\left\lceil n/2 \right\rceil }}{t_1,t_2,\ldots,t_{\left\lceil n/2 \right\rceil }},\ n>0, 
\end{equation}
where $a_k=2k-1$ and $\lceil x \rceil$ stands for the smallest integer not less than $x$.
Therefore,  for any positive integer $n$ we have
\begin{equation}\label{eq4}
q(n)\leqslant F_n.
\end{equation}

The inequality \eqref{eq4} also follows from the fact that the number of compositions of $n$ into odd parts equals $F_n$ \cite[Exercise 35 (d), pg. 121]{Stanleyec1}.

In this paper, we  use the method presented in \cite{Merca14b} to introduce new upper bounds involving Fibonacci numbers for the number of partitions of $n$ into distinct parts. As corollaries of this fact, we derive several infinite families of  inequalities involving the generating function for the number of odd divisors. We demonstrate the usefulness of these inequalities in several examples in which we find explicit upper bounds for $$\sum_{n=1}^{\infty} \frac{4^n}{4^{2n}-1}$$ and $$\prod_{n=1}^{\infty} \left( 1+ \frac{1}{4^n}\right).$$ Such products, in particular, have been of interest in number theory of a long time. As far as we know, there are no analytic proofs for bounds of the type given in this article.

\section{On the multinomial coefficients over partitions into odd parts}

We denote by $Q_k(n)$ the number of multinomial coefficients such that
$$\binom{t_1+t_2+\cdots+t_{\left\lceil n/2 \right\rceil }}{t_1,t_2,\ldots,t_{\left\lceil n/2 \right\rceil }}=k$$
and
$$\sum_{k=1}^{\lceil n/2 \rceil}(2k-1)t_k=n.$$
Then, 
\begin{equation}\label{eq2.0a}
\sum_{k\geqslant 1}Q_k(n)=q(n)
\end{equation}
and, from  \eqref{eq3}, we have
\begin{equation}\label{eq2.0b}
\sum_{k\geqslant 1}kQ_k(n)=F_n.
\end{equation}
We see that $Q_1(n)$ is equal to the number of odd positive divisors of $n$. Thus, if $\tau(n)$ denotes the number of divisions of $n$, we have
$$Q_1(n)=\begin{cases}
\tau(n), & \text{if $n$ is odd,}\\
\tau(n)-\tau(n/2), & \text{if $n$ is even}
\end{cases}$$
and
\begin{equation}\label{eq2.1}
Q_2(n)=\begin{cases}
0, & \text{if $n$ is odd,}\\
\left\lfloor \dfrac{n}{4} \right\rfloor, & \text{if $n$ is even,}
\end{cases}\\
\end{equation}
where $\lfloor x \rfloor$ stands for the greatest integer less than or equal to $x$.

Notice that $Q_2(n)$ equals the number of partitions of $n$ with two distinct odd parts.

On the other hand, following Guo \cite[Theorem 1]{Guo14}, it is an easy exercise to derive the following formula for $Q_k(n)$, when $k$ is a power of a prime. 

\begin{theorem}\label{Th1}
Let $p$ be an odd prime and let $n$, $r$ be positive integers. Then
$$Q_{p^r}(n)=\begin{cases}
\left\lceil \frac{1}{2} \left\lfloor \frac{n-1}{p^r-1} \right\rfloor \right\rceil - \delta_{0,n \bmod p^r},& \text{if $n$ is odd,}\\
0, & \text{if $n$ is even.}
\end{cases}
$$
For $r>1$,
$$Q_{2^r}(n)=\begin{cases}
0, & \text{if $n$ is odd,}\\
\left\lceil \frac{1}{2} \left\lfloor \frac{n-1}{2^r-1} \right\rfloor \right\rceil - \delta_{0,n \bmod 2^r}\cdot \left\lfloor \frac{n}{2^r} \right\rfloor \bmod 2,& \text{if $n$ is even,}
\end{cases}
$$
where $\delta_{i,j}$ is the Kronecker delta.
\end{theorem}

It is well known that the generating function for the number of odd positive divisors is given by
\begin{equation}\label{Q1gen}
\sum_{n=1}^{\infty}Q_1(n)x^n=\sum_{n=1}^{\infty}\frac{x^n}{1-x^{2n}},\qquad \left|x\right|<1.
\end{equation}
On the other hand, we have
\begin{equation}\label{Q2gen}\sum_{n=1}^{\infty}Q_2(n)x^n
=\sum_{n=1}^{\infty}\left\lfloor \frac{n}{2} \right\rfloor x^{2n}
= \frac{x^4}{(1-x^2)(1-x^4)},\qquad \left|x\right|<1.\end{equation} 


The next theorem gives the  generating function for $\{Q_k(n)\}_{n\geq1}$, when $k$ is a power of a prime. 

\begin{theorem}\label{Th2}
Let $p$ be a prime and let $r$ be a positive integer with $p^r>2$. The generating function of $Q_{p^r}(n)$ is given by
$$\sum_{n=1}^{\infty}Q_{p^r}(n)x^n=\frac{x^{p^r}}{(1-x^2)\left(1-x^{2(p^r-1)} \right) }-\frac{x^{p^r}}{1-x^{2p^r}},\qquad \left|x\right|<1.$$
\end{theorem}

\begin{proof}
For all positive integers $n$ and $k$, we have
$$\left\lceil \frac{1}{2} \left\lfloor \frac{n}{k} \right\rfloor \right\rceil = \left\lfloor \frac{n+k}{2k} \right\rfloor.$$
When $p$ is an odd prime and $r$ is a positive integer, this identity and Theorem \ref{Th1} allow us to write
\begin{eqnarray*}
\sum_{n=1}^{\infty}Q_{p^r}(n)x^n
&=& \sum_{n=1}^{\infty} \left\lceil \frac{1}{2} \left\lfloor \frac{2n-2}{p^r-1} \right\rfloor \right\rceil x^{2n-1}-\sum_{n=1}^{\infty} x^{p^r(2n-1)}\\
&=& \sum_{n=1}^{\infty} \left\lfloor \frac{2n-2+p^r-1}{2(p^r-1)} \right\rfloor  x^{2n-1}-\frac{1}{x^{p^r}}\sum_{n=1}^{\infty} x^{2p^rn}\\
&=& \sum_{n=1}^{\infty} \left\lfloor \frac{n-1}{p^r-1} \right\rfloor  x^{2n-p^r}-\frac{1}{x^{p^r}}\left( \frac{1}{1-x^{2p^r}}-1\right)\\
&=&\frac{1}{x^{p^r-2}} \sum_{n=0}^{\infty} \left\lfloor \frac{n}{p^r-1} \right\rfloor  x^{2n}-\frac{x^{p^r}}{1-x^{2p^r}}\\
&=&\frac{1}{x^{p^r-2}}\cdot \frac{x^{2(p^r-1)}}{(1-x^2)(1-x^{2(p^r-1)})}-\frac{x^{p^r}}{1-x^{2p^r}},\qquad \left|x \right|<1.
\end{eqnarray*}
The last identity follows from the generating function of $\left\lfloor n/k \right\rfloor$,
$$\sum_{n=0}^{\infty}\left\lfloor \frac{n}{k}\right\rfloor x^n
=\frac{x^k}{(1-x)(1-x^k)},\qquad \left|x \right|<1,$$
where $k$ is a positive integer.

Now let $p=2$. If $n \bmod 2^r=0$ and $\left\lfloor \frac{n}{2^r} \right\rfloor \bmod 2=1$, we must have $n=2^r k$ with $k$ odd.
Thus, in the case $p=2$ and $r>1$, using Theorem \ref{Th1}, we have
\begin{eqnarray*}
\sum_{n=1}^{\infty}Q_{2^r}(n)x^n
&=& \sum_{n=1}^{\infty} \left\lceil \frac{1}{2} \left\lfloor \frac{2n-1}{2^r-1} \right\rfloor \right\rceil x^{2n}-\sum_{n=1}^{\infty} x^{2^r(2n-1)}\\
&=& \sum_{n=1}^{\infty} \left\lfloor \frac{2n-1+2^r-1}{2(2^r-1)} \right\rfloor  x^{2n}-\frac{1}{x^{2^r}}\sum_{n=1}^{\infty} x^{2^{r+1}n}\\
&=& \sum_{n=1}^{\infty} \left\lfloor \frac{n-1}{2^r-1} \right\rfloor  x^{2n-2^r}-\frac{1}{x^{2^r}}\left( \frac{1}{1-x^{2^{r+1}}}-1\right)\\
&=&\frac{1}{x^{2^r-2}} \sum_{n=0}^{\infty} \left\lfloor \frac{n}{2^r-1} \right\rfloor  x^{2n}-\frac{x^{2^r}}{1-x^{2^{r+1}}}\\
&=&\frac{1}{x^{2^r-2}}\cdot \frac{x^{2^{r+1}-2}}{(1-x^2)(1-x^{2^{r+1}-2})}-\frac{x^{2^r}}{1-x^{2^{r+1}}},
\end{eqnarray*}
This concludes the proof.
\end{proof}

As in \cite[Theorem 2]{Merca14b}, we obtain

\begin{theorem}\label{Th3}
For $n,k>0$,
$$q(n) \leqslant \frac{1}{k}\left(  F_n+\sum_{j=1}^{k-1}(k-j)Q_j(n)\right) .$$
\end{theorem}

\section{A family of double inequalities}

Recall that the generating function for the number of partitions into distinct parts is given by
\begin{equation}\label{eq3.1}
\sum_{n=0}^{\infty}q(n)x^n
=(-x;x)_{\infty},
\end{equation}
where
$$(a;q)_n=(1-a)(1-aq)(1-aq^2)\cdots(1-aq^{n-1})$$
is  the $q$-shifted factorial with $(a;q)_0=1$. The generating function of the Fibonacci sequence is the power series
$$\sum_{n=0}^{\infty}F_nx^n.$$
We note \cite{Glaister95} that this series is convergent for $\left|x \right|<\varphi-1$, where $\varphi$ is the golden ratio, i.e.,
$$\varphi=\frac{1+\sqrt{5}}{2},$$
and its sum has a simple closed-form:
\begin{equation}\label{eq3.2}
\sum_{n=0}^{\infty} F_nx^n=\frac{x}{1-x-x^2}.
\end{equation}

For $n>0$, we have two trivial inequalities
$$Q_1(n)\leqslant q(n)\qquad\text{and}\qquad q(n)\leqslant F_n.$$
These inequalities allow us to derive
\begin{equation}\label{eq3.3}
1+\sum_{n=1}^{\infty}\frac{x^n}{1-x^{2n}}<(-x;x)_{\infty},\qquad 0<x<1,
\end{equation}
and
\begin{equation}\label{eq3.4}
(-x;x)_{\infty}<\frac{1-x^2}{1-x-x^2},\qquad 0<x<\varphi-1.
\end{equation}
From these inequalities, with $x$ replaced by $1/4$, we obtain the double inequality
$$
1+\sum_{n=1}^{\infty}\frac{4^{n}}{4^{2n}-1} 
< \prod_{n=1}^{\infty}\left(1+\frac{1}{4^{n}} \right)
< \frac{15}{11}.
$$

In this section, we show how to improve the inequalities \eqref{eq3.3} and \eqref{eq3.4}. To do this, we denote by $P$ the set of positive integers greater than 2 whose divisors form a geometric progression, i.e.,
\begin{eqnarray*}
P &=& \{p^r | \text{$p$ prime and $r$ positive integer with $p^r>2$ } \}\\
&=& \{ 3,4,5,7,8,9,11,13,16,17,19,23,25,27,29,31,32,\ldots \}.
\end{eqnarray*}
Using  equation \eqref{eq2.0a} and Theorem \ref{Th2}, we can improve the inequality given in \eqref{eq3.3}.

\begin{corollary}\label{C1}
For $0<x<1$,
\begin{align*}
& 1+\sum_{n=1}^{\infty}\frac{x^n}{1-x^{2n}}<
(-x;x)_{\infty}-\frac{x^4}{(1-x^2)(1-x^4)}\\
&\qquad\qquad\qquad\qquad\qquad-\sum_{k\in P}\left(\frac{x^k}{(1-x^2)(1-x^{2(k-1)})}-\frac{x^k}{1-x^{2k}} \right).
\end{align*}
\end{corollary}
 
 From Corollary \ref{C1},  replacing $x$ by $1/4$ and 
 considering only the term for $k=3$ in the last sum, we obtain
\begin{equation}\label{eq3.5}
\frac{69983}{69615}+\sum_{n=1}^{\infty}\frac{4^{n}}{4^{2n}-1} 
< \prod_{n=1}^{\infty}\left(1+\frac{1}{4^{n}} \right).
\end{equation}

The  inequality in the next corollary follows easily from equation \eqref{eq2.0b} and Theorem \ref{Th2}.

\begin{corollary}\label{C2}
For $0<x<\varphi-1$,
\begin{align*}
& \sum_{n=1}^{\infty}\frac{x^n}{1-x^{2n}}<
\frac{x}{1-x-x^2}-\frac{2x^4}{(1-x^2)(1-x^4)}\\
&\qquad\qquad\qquad\qquad\qquad-\sum_{k\in P}\left(\frac{kx^k}{(1-x^2)(1-x^{2(k-1)})}-\frac{kx^k}{1-x^{2k}} \right).
\end{align*}
\end{corollary}

From Corollary \ref{C2},  replacing $x$ by $1/4$ and 
 considering only the term for $k=3$ in the last sum, we obtain
$$
\sum_{n=1}^{\infty}\frac{4^{n}}{4^{2n}-1} 
<\frac{1347596}{3828825}.
$$

Other inequalities can be obtained considering a few cases of Theorem \ref{Th3}. For example, the case $k=2$ of this theorem reads as
$$q(n)\leqslant \frac{F_n+Q_1(n)}{2},\qquad n>0.$$
Thus,  for $0<x<\varphi-1$, we derive the inequality
\begin{equation}\label{eq3.5a}
(-x;x)_{\infty}<1+\frac{x}{2(1-x-x^2)}+\frac{1}{2}\sum_{n=1}^{\infty}\frac{x^n}{1-x^{2n}}.
\end{equation}
Replacing $x$ by $1/4$, we obtained
\begin{equation}\label{eq3.5b}
\prod_{n=1}^{\infty}\left(1+\frac{1}{4^n} \right)<\frac{13}{11}+\frac{1}{2}\sum_{n=1}^{\infty}\frac{4^n}{4^{2n}-1}.
\end{equation}
For $k=3$ we have
$$q(n)\leqslant \frac{F_n+2Q_1(n)+Q_2(n)}{3},\qquad n>0,$$
Then, for $0<x<\varphi-1$, we obtain
\begin{equation}\label{eq3.5c}
(-x;x)_{\infty}<1+\frac{x}{3(1-x-x^2)}+\frac{x^4}{3(1-x^2)(1-x^4)}+\frac{2}{3}\sum_{n=1}^{\infty}\frac{x^n}{1-x^{2n}}.
\end{equation}
The case $x=1/4$ of this inequality and the inequality \eqref{eq3.5} allow us to write 
\begin{equation}\label{3.last}\frac{69983}{69615}+\sum_{n=1}^{\infty}\frac{4^{n}}{4^{2n}-1} 
<\prod_{n=1}^{\infty}\left(1+\frac{1}{4^n} \right)<\frac{141701}{126225}+\frac{2}{3}\sum_{n=1}^{\infty}\frac{4^n}{4^{2n}-1}.\end{equation}

The statements of Corollaries \ref{C1} and \ref{C2} can be summarized as a family of double inequalities involving generating functions for the number of odd divisors, i.e.,


$$A_k(x)+\sum_{n=0}^{\infty}\frac{x^n}{1-x^{2n}}
<\prod_{n=1}^{\infty}\left(1+x^n\right)
<B_k(x)+\frac{k-1}{k}\sum_{n=0}^{\infty}\frac{x^n}{1-x^{2n}},$$
with $0<x<\varphi-1$.

For $0<x<1$, the sequence $\{ A_k(x)\}_{k>0}$ is strictly monotonically increasing and its terms are given by Corollary \ref{C1}, i.e.,
$$A_1(x)=1,\qquad A_2(x)=A_1(x)+\frac{x^4}{(1-x^2)(1-x^4)},$$
and
$$A_k(x)=A_{k-1}(x)+\frac{x^{p_k}}{(1-x^2)(1-x^{2(p_k-1)})}-\frac{x^{p_k}}{1-x^{2p_k}},$$
where $p_k$ is $k$th positive integer whose divisors form a geometric progression.
 
For $0<x<\varphi-1$, the sequence $\{ B_k(x)\}_{k>0}$ is strictly monotonically decreasing and its terms are given by Theorem  \ref{Th3}, 
$$B_1(x)=1+\frac{x}{1-x-x^2}$$
and
$$kB_k(x)=(k-1)B_{k-1}(x)+1+\sum_{j=2}^{k-1}GF_j(x),$$
where $GF_k(x)$ is the generating function for $Q_k(n)$.

It is still an open problem to give the expression of the generating function for $Q_k(n)$ when $k$ is a positive integer whose divisors do not form a geometric progression.

\section{Upper bounds for $(-x;x)_{\infty}$}

Motivated 
by the results presented in the previous section, we introduce new upper bounds for the generating function for the number of odd divisors of positive integers. We use the recurrence relation \eqref{eq2} and the inequality given by the following theorem. 

\begin{theorem}\label{thm4} For any positive integer $n$, we have \begin{equation}\label{q-ineq} q(n)-q(n-1)-q(n-2)\leqslant 0,\quad n>0\end{equation}

\end{theorem}

\begin{proof}
Taking into account the identity
$$(-x;x)_{\infty}=\frac{1}{(x;x^2)_{\infty}},$$
we need to show that, the coefficient of $x^n$ in
$$\frac{1-x-x^2}{(x;x^2)_{\infty}}$$
is non-positive for $n>0$. Using Euler's formula \cite[p. 19, Eq.(2.2.5)]{Andrews76}
$$\sum_{n=0}^{\infty}\frac{z^n}{(q;q)_n}=\frac{1}{(z;q)}_{\infty},\qquad \left| z \right|, \left| q \right|<1,  $$ we have
\begin{eqnarray*}
\frac{1-x-x^2}{(x;x^2)_{\infty}}
&=& \frac{1}{(x^3;x^2)_{\infty}}-\frac{x^2}{(1-x)(x^3;x^2)_{\infty}} \\
&=& \sum_{n=0}^{\infty}\frac{x^{3n}}{(x^2;x^2)_{n}}
-\sum_{n=0}^{\infty}\frac{x^{3n+2}}{(1-x)(x^2;x^2)_{n}}  \\
&=& 1+\sum_{n=1}^{\infty}\frac{x^{3n}}{(x^2;x^2)_{n}}
-\sum_{n=1}^{\infty}\frac{x^{3n-1}(1-x^{2n})}{(1-x)(x^2;x^2)_{n}}  \\
&=& 1+\sum_{n=1}^{\infty}\frac{x^{3n-1}}{(x^2;x^2)_{n}}\left( x
-\frac{1-x^{2n}}{1-x} \right)  \\
&=& 1+\sum_{n=1}^{\infty}\frac{x^{3n-1}}{(x^2;x^2)_{n}}\left( x
-(1+x+x^2+\cdots+x^{2n-1}) \right)  \\
&=& 1-\sum_{n=1}^{\infty}\frac{x^{3n-1}}{(x^2;x^2)_{n}}\left( 
1+x^2+\cdots+x^{2n-1} \right) ,
\end{eqnarray*}
Then, for $n>0$, the coefficient of $x^n$ is non-positive.
\end{proof}

Theorem \ref{thm4} can also be proven combinatorially as follows. \medskip

\noindent \textit{Proof 2.} To see that \eqref{q-ineq} holds, consider the following injections. Each partition of $n$ with odd parts and smallest part equal to $1$ corresponds to a partition of $n-1$ with odd parts (by removing a part equal to $1$).  Moreover,  each partition of $n$ with odd parts and smallest part at least  $3$ corresponds to a partition  of $n-2$ with odd parts (by subtracting $2$ from the smallest part). \hfill \ \hfill $\Box$\medskip

Clearly,  the inequality
\begin{equation}\label{eq4.1}
q(n-1)+q(n-2)\leqslant F_n,\qquad n>2
\end{equation}
is stronger than the trivial inequality \eqref{eq4}. On the other hand, we have the identity
$$F_{n-2}+2F_{n-3}+F_{n-4}=F_n,$$
which allows us to derive the following inequality
\begin{equation}\label{eq4.2}
q(n-2)+2q(n-3)+q(n-4)\leqslant F_n,\qquad n>4.
\end{equation}
Finally, after $k$ steps we obtain
\begin{proposition}\label{C4.1}
Let $k$ and $n$ be two nonnegative integers such that $n>2k$. Then
\begin{equation}\label{eq-prop1}\sum_{j=0}^{k}\binom{k}{j}q(n-k-j)\leqslant F_n.\end{equation}
\end{proposition}

The proposition gives an infinite family of inequalities. The  inequalities  \eqref{eq4}, \eqref{eq4.1}, and \eqref{eq4.2} correspond to \eqref{eq-prop1} with  $k=0,1,$ and $2$ respectively.

The next theorem introduces an infinite family of upper bounds for $(-x;x)_{\infty}$, when $0<x<\varphi-1$. 

\begin{theorem}\label{th4}
Let $k$ be a nonnegative integer. For $0<x<\varphi-1$,
$$(-x;x)_{\infty}
<\frac{1}{x^{k-1}(1+x)^k(1-x-x^2)}
-\frac{1}{x^k(1+x)^k}\sum_{n=0}^{2k}\left(F_n-S_{n,k} \right)x^n,$$
where 
$$S_{n,k}=\sum_{j=0}^{k}\binom{k}{j}q(n-k-j).$$
\end{theorem}

\begin{proof} We first derive the  generating function for $S_{n,k}$.  Using the well-known Cauchy products of two power series, we have $$\left(\sum_{j=0}^{k}\binom{k}{j}x^j \right) \left(\sum_{n=k}^{\infty}q(n-k)x^n \right)=\sum_{n=0}^{\infty}\left(\sum_{j=0}^{k}\binom{k}{j}q(n-k-j) \right)x^n.  $$ Therefore, $$\sum_{n=0}^{\infty}S_{n,k}x^n=x^k(1+x)^k(-x;x)_{\infty},\qquad \left|x \right|<1.$$
Then, using  Proposition  \ref{C4.1}, for $0<x<\varphi-1$, we can write
$$x^k(1+x)^k(-x;x)_{\infty}-\sum_{n=0}^{2k}S_{n,k}x^n<\frac{x}{1-x-x^2}-\sum_{n=0}^{2k}F_nx^n,$$
and the proof of the theorem follows.
\end{proof}

From Theorem \ref{th4}, for each $k\in \{0,1,2,3\}$ we obtain the inequality \eqref{eq3.4}. The case $k=4$ of this theorem read as
$$(-x;x)_{\infty}=\prod_{n=1}^{\infty}(1+x^n)<\frac{1+4x+5x^2-6x^4-3x^5}{(1+x)^4(1-x-x^2)},\qquad 0<x<\varphi-1.$$
Replacing $x$ by $1/4$ in the last inequality, we obtain 
$$\prod_{n=1}^{\infty}\left( 1+\frac{1}{4^n} \right) <\frac{9364}{6875}=1.362036364...\ .$$ 
For $k=5$, we obtain
$$\prod_{n=1}^{\infty}(1+x^n)<\frac{1+5x+9x^2+5x^3-6x^4-15x^5+3x^6+6x^7}{(1+x)^5(1-x-x^2)},\qquad 0<x<\varphi-1.$$
From this inequality, with $x$ replaced by $1/4$, we obtain
$$\prod_{n=1}^{\infty}\left( 1+\frac{1}{4^n} \right) <\frac{46754}{34375}=1.360116364...\ .$$

Note that using a partial sum with $16$ terms in \eqref{3.last},  we obtain a lower bound for $\left( -\frac{1}{4};\frac{1}{4}\right) _{\infty}$ of $1.35553519$. 
Moreover, if $k=6$ in Theorem \ref{th4}, 
$$\prod_{n=1}^{\infty}\left( 1+x^n \right) <\frac{30x^8+33x^7-36x^6-21x^5-x^4+14x^3+14x^2+6x+1}{(1-x-x^2)(1+x)^6}$$ and for $x=1/4$ we obtain 
$$\prod_{n=1}^{\infty}\left( 1+\frac{1}{4^n} \right) <\frac{233506}{171875}=1.3585803636...\ .$$

We denote by $R_k(x)$ the right hand side of the inequality of Theorem \ref{th4} and show  that $\{R_k(x)\}_{k \geqslant 0}$ is a decreasing sequence for $x >0$. First we prove a helpful lemma. 

\begin{lemma} \label{th5}
Let $k$ and $n$ be two nonnegative. Then
$$S_{n+2,k+1}=S_{n+1,k}+S_{n,k}$$
\end{lemma}

\begin{proof} We have
\begin{align*}
&S_{n+2,k+1}-S_{n+1,k}\\
&\qquad=\sum_{j=0}^{k+1} \binom{k+1}{j}q(n+1-k-j)- \sum_{j=0}^{k} \binom{k}{j}q(n+1-k-j)\\
&\qquad=\sum_{j=0}^{k} \left(\binom{k+1}{j}-\binom{k}{j}\right)q(n+1-k-j)+q(n-2k)\\
&\qquad=\sum_{j=1}^{k} \binom{k}{j-1}q(n+1-k-j)+q(n-2k)\\
&\qquad=\sum_{j=0}^{k-1} \binom{k}{j}q(n-k-j)+q(n-2k)\\
&\qquad=\sum_{j=0}^{k} \binom{k}{j}q(n-k-j)\\
&\qquad=S_{n,k}. 
\end{align*}
\end{proof}

\begin{theorem}\label{decr} 
Let $k$ be a nonnegative integer. For $x>0$, $$R_{k+1}(x)-R_k(x) \leqslant 0.$$
\end{theorem}

\begin{proof} After some simple manipulations, we see that the inequality in the statement of the theorem 
is equivalent to  
$$L_k(x)=x+x(1+x)\sum_{n=0}^{2k}\left(F_n-S_{n,k} \right)x^n  -\sum_{n=0}^{2k+2}\left(F_n-S_{n,k+1} \right)x^n\leqslant 0.$$ 
Then, using Lemma \ref{th5}, we obtain
$$
L_k(x)=S_{0,k+1}+(S_{1,k+1}-S_{0,k})x+(F_{2k}-F_{2k+2}+S_{2k+2,k+1}-S_{2k,k})x^{2k+2}.
$$
Note  that $S_{0,k+1}=0$ and $S_{0,k}=S_{1,k+1}$ for all $k\geqslant 0$. 
Then, by Lemma \ref{th5} again, we have
$$L_k(x)=(S_{2k+1,k}-F_{2k+1})x^{2k+2}.$$ 
The result now follows from Proposition \ref{C4.1} with $n=2k+1$.
\end{proof}

\section{Further inequalities involving the generating functions for $q(n), Q_1(n)$, and $Q_2(n)$}

In this section, we can proceed in analogy with the theorems of the previous section to obtain new families of inequalities. We make use of the identity$$\sum_{j=0}^{k}\binom{k}{j}F_{n-k-j}=F_n$$ and 
 the special cases  of Theorem \ref{Th3} when $k=2, 3$, i.e.,  $$2q(n)-Q_1(n) \leqslant F_n,$$ $$3q(n)-2Q_1-Q_2(n)\leqslant F_n.$$

The proofs of the next two theorems are similar to the proof of Theorem \ref{th4}.

\begin{theorem}\label{th6}
Let $k$ be a nonnegative integer. For $0<x<\varphi-1$,
$$2(-x;x)_{\infty}-\sum_{n=1}^{\infty}\frac{x^n}{1-x^{2n}}
<\frac{1}{x^{k-1}(1+x)^k(1-x-x^2)}
-\frac{1}{x^k(1+x)^k}\sum_{n=0}^{2k}\left(F_n-S_{n,k} \right)x^n,$$
where 
$$S_{n,k}=2\sum_{j=0}^{k}\binom{k}{j}q(n-k-j)-\sum_{j=0}^{k}\binom{k}{j}Q_1(n-k-j),$$
with $Q_1(n)=0$ for any negative $n$.
\end{theorem}

For each $k\in\{0,1,2,3,4\}$, Theorem \ref{th6} reduces to  inequality \eqref{eq3.5a}.
The case $k=5$ gives
$$2(-x;x)_{\infty}-\sum_{n=1}^{\infty}\frac{x^n}{1-x^{2n}}
<\frac{2+9x+13x^2-20x^4-24x^5-10x^6-x^7}{(1+x)^5(1-x-x^2)}$$
for $0<x<\varphi-1$.
Replacing $x$ by $1/4$, we obtain
$$\prod_{n=1}^{\infty}\left(1+\frac{1}{4^n} \right)
<\frac{81239}{68750}+\frac{1}{2}\sum_{n=1}^{\infty}\frac{4^n}{4^{2n}-1}
=2.363316364\ldots+\frac{1}{2}\sum_{n=1}^{\infty}\frac{4^n}{4^{2n}-1}
.$$
This is an improved version of the inequality \eqref{eq3.5b}.
If $k=6$ and $0<x<\varphi-1$,  we have
$$2(-x;x)_{\infty}-\sum_{n=1}^{\infty}\frac{x^n}{1-x^{2n}}
<\frac{2+11x+22x^2+13x^3-20x^4-44x^5-41x^6-4x^7+6x^8}{(1+x)^6(1-x-x^2)},$$
and for $x=1/4$ we obtain
$$\prod_{n=1}^{\infty}\left(1+\frac{1}{4^n} \right)
<\frac{406118}{171875}+\frac{1}{2}\sum_{n=1}^{\infty}\frac{4^n}{4^{2n}-1}
=2.362868364\ldots+\frac{1}{2}\sum_{n=1}^{\infty}\frac{4^n}{4^{2n}-1}.$$

\begin{theorem}\label{th7}
Let $k$ be a nonnegative integer. For $0<x<\varphi-1$,
\begin{eqnarray*}
3(-x;x)_{\infty}-2\sum_{n=1}^{\infty}\frac{x^n}{1-x^{2n}}
&<&\frac{x^4}{(1-x^2)(1-x^4)}+\frac{1}{x^{k-1}(1+x)^k(1-x-x^2)}\\
& &\qquad -\frac{1}{x^k(1+x)^k}\sum_{n=0}^{2k}\left(F_n-S_{n,k} \right)x^n,
\end{eqnarray*}
where 
$$S_{n,k}=\sum_{j=0}^{k}\binom{k}{j}(3q(n-k-j)-2Q_1(n-k-j)-Q_2(n-k-j)),$$
with $Q_1(n)=Q_2(n)=0$ for any negative $n$.
\end{theorem}

For each $k<6$, Theorem \ref{th7} reduces to  inequality \eqref{eq3.5c}. The case $k=6$  reads as
\begin{align*}
& 3(-x;x)_{\infty}-2\sum_{n=1}^{\infty}\frac{x^n}{1-x^{2n}}<\frac{x^4}{(1-x^2)(1-x^4)}\\
&\qquad\qquad\qquad+\frac{3+16x+30x^2+12x^3-40x^4-72x^5-55x^6-19x^7-2x^8}{(1+x)^6(1-x-x^2)},
\end{align*}
for $0<x<\varphi-1$.
Replacing $x$  by $1/4$ in this inequality, we obtain
$$\prod_{n=1}^{\infty}\left(1+\frac{1}{4^n} \right)
<\frac{88561442}{78890625}+\frac{2}{3}\sum_{n=1}^{\infty}\frac{4^n}{4^{2n}-1}
=1.122585123\ldots+\frac{2}{3}\sum_{n=1}^{\infty}\frac{4^n}{4^{2n}-1}.$$

\section{Concluding remarks}

Using the representation of Fibonacci number as a sum of multinomial coefficients over partitions with odd parts, we obtain several infinite families of double inequalities involving the generating function of the number of partitions with odd parts,  the generating function for the number of odd divisors and the generating function of the number of partitions in two distinct odd parts. We note that when the problem of finding a closed form for the generating function of $Q_k(n)$ for arbitrary $k$ will be solved, then further, stronger inequality families will follow by the methods used in this article. 

\section*{Acknowledgements} The second author wishes to thank the College of the Holy Cross for its hospitality during the writing of this article.



\bigskip


\end{document}